\documentclass[10pt]{amsart}
\usepackage[utf8]{inputenc}
\usepackage{todonotes}

\usepackage{etoolbox}

\makeatletter
\patchcmd{\@thm}{\let\thm@indent\indent}{\let\thm@indent\noindent}{}{}
\patchcmd{\@thm}{\thm@headfont{\scshape}}{\thm@headfont{\bfseries}}{}{}

\usepackage[T1]{fontenc}

\usepackage{hyperref}
\hypersetup{
    colorlinks = true,
    citecolor = [rgb]{0.00, 0.5, 0.00},
}

\usepackage{mathtools}

\usepackage{multirow}
\usepackage{longtable}

\usepackage{amsmath}
\usepackage{amsfonts}
\usepackage{amssymb}
\usepackage{bbm}
\usepackage{float}
\usepackage{amsthm}
\usepackage{comment}
\usepackage{amscd}
\usepackage{amsxtra}
\usepackage{epsfig}
\usepackage{epigraph}
\usepackage{pgfplots}

\usepackage{stmaryrd}

\usepackage{listings}
\usepackage{color}

\definecolor{dkgreen}{rgb}{0,0.6,0}
\definecolor{gray}{rgb}{0.5,0.5,0.5}
\definecolor{mauve}{rgb}{0.58,0,0.82}

\usepackage{color}
\usepackage[all]{xy}
\usepackage{verbatim}

\usepackage{eucal}
\usepackage{mathrsfs}

\usepackage{graphicx}
\usepackage{tikz-cd}

\usepackage{url}
\usepackage{verbatim}

\usepackage{xy}
\xyoption{all}

\usepackage{amsthm}
\usepackage{hhline}

\usepackage{caption} 
\makeatletter
\def\@tocline#1#2#3#4#5#6#7{\relax
  \ifnum #1>\c@tocdepth 
  \else
    \par \addpenalty\@secpenalty\addvspace{#2}%
    \begingroup \hyphenpenalty\@M
    \@ifempty{#4}{%
      \@tempdima\csname r@tocindent\number#1\endcsname\relax
    }{%
      \@tempdima#4\relax
    }%
    \parindent\z@ \leftskip#3\relax \advance\leftskip\@tempdima\relax
    \rightskip\@pnumwidth plus4em \parfillskip-\@pnumwidth
    #5\leavevmode\hskip-\@tempdima
      \ifcase #1
       \or\or \hskip 1em \or \hskip 2em \else \hskip 3em \fi%
      #6\nobreak\relax
    \hfill\hbox to\@pnumwidth{\@tocpagenum{#7}}\par
    \nobreak
    \endgroup
  \fi}
\makeatother

\usepackage[noabbrev,capitalize]{cleveref} 

\crefformat{nul}{(#2#1#3)}
\Crefformat{nul}{(#2#1#3)}

\crefname{section}{\S}{\S\S}
\crefname{subsection}{\S}{\S\S}
\crefname{axioms}{Axiom}{Axioms}
\crefname{exercise}{Exercise}{Exercises}
\crefname{exercisenum}{Exercise}{Exercises}
\crefname{construction}{Construction}{Constructions}
\crefname{problem}{Problem}{Problems}
\crefname{theorem}{Theorem}{Theorems}
\crefname{definition}{Definition}{Definitions}
\crefname{prop}{Proposition}{Propositions}
\crefname{lemma}{Lemma}{Lemmas}
\crefname{example}{Example}{Examples}
\crefname{examplealph}{Example}{Examples}
\crefname{corollary}{Corollary}{Corollaries}
\crefname{nonexample}{Nonexample}{Nonexamples}
\crefname{equation}{}{}
\crefname{summary}{Summary}{Summaries}
\crefname{recollection}{Recollection}{Recollections}
\Crefname{recollection}{Recollection}{Recollections}
\Crefname{nonexample}{Nonexample}{Nonexamples}
\Crefname{corollary}{Corollary}{Corollaries}
\Crefname{corollary}{Corollary}{Corollaries}
\Crefname{axioms}{Axiom}{Axioms}
\Crefname{exercise}{Exercise}{Exercises}
\Crefname{exercisenum}{Exercise}{Exercises}
\Crefname{construction}{Construction}{Constructions}
\Crefname{problem}{Problem}{Problems}
\Crefname{theorem}{Theorem}{Theorems}
\Crefname{definition}{Definition}{Definitions}
\Crefname{prop}{Proposition}{Propositions}
\Crefname{lemma}{Lemma}{Lemmas}
\Crefname{example}{Example}{Examples}
\Crefname{examplealph}{Example}{Examples}
\Crefname{section}{\S}{\S\S}
\Crefname{subsection}{\S}{\S\S}

\DeclareMathOperator{\End}{End}

\newtheorem{lemma}{Lemma}[section]

\newtheorem{corollary}[lemma]{Corollary}

\newtheorem{theorem}[lemma]{Theorem}

\newtheorem{prop}[lemma]{Proposition}

\newtheorem*{conj-moore}{Conjecture~\ref{moore-splitting}}
\newtheorem*{conj-cent}{Conjecture~\ref{centrality-conj}}
\newtheorem*{conj-tmf}{Conjecture~\ref{tmf-conj}}
\newtheorem*{thm-main}{Theorem~\ref{main-thm}}
\newtheorem*{cor-bpn}{Corollary~\ref{bpn}}
\newtheorem*{thm-string}{Theorem~\ref{mstring}}
\newtheorem{innercustomthm}{Theorem}
\newenvironment{customthm}[1]
{\renewcommand\theinnercustomthm{#1}\innercustomthm}
{\endinnercustomthm}

\theoremstyle{definition}

\newtheorem{definition}[lemma]{Definition}

\newtheorem{example}[lemma]{Example}

\newtheorem{remark}[lemma]{Remark}

\newcommand{\crefaliaslemma}[1]{\AddToHook{env/#1/begin}{\crefalias{lemma}{#1}}}
\forcsvlist{\crefaliaslemma}{corollary,prthm,theorem,prop,conjecture,theoremno,question,assume,definition,warning,example,problem,construction,remark,notation,variant,setup,terminology,recall,thesis,exercise,slogan,observe,summary,recipe}
\crefname{prthm}{Proto-Theorem}{Proto-Theorems}
\crefname{conjecture}{Conjecture}{Conjectures}
\crefname{theoremno}{Theorem}{Theorems}
\crefname{question}{Question}{Questions}
\crefname{assume}{Assumption}{Assumptions}
\crefname{warning}{Warning}{Warnings}
\crefname{remark}{Remark}{Remarks}
\crefname{notation}{Notation}{Notations}
\crefname{variant}{Variant}{Variants}
\crefname{setup}{Setup}{Setups}
\crefname{terminology}{Terminology}{Terminologies}
\crefname{recall}{Recollection}{Recollections}
\crefname{thesis}{Thesis}{Theses}
\crefname{slogan}{Slogan}{Slogans}
\crefname{observe}{Observation}{Observations}
\crefname{recipe}{Recipe}{Recipes}

\renewcommand{\AA}{\mathbf{A}}
\newcommand{\FF}{\mathbf{F}}
\newcommand{\bW}{\mathbb{W}}
\newcommand{\Z}{\mathbf{Z}}
\newcommand{\QQ}{\mathbf{Q}}
\newcommand{\ZZ}{\mathbf{Z}}

\newcommand{\cO}{\mathcal{O}}

\newcommand{\RR}{\mathbf{R}}

\newcommand{\Sp}{\mathrm{Sp}}

\newcommand{\SP}{\mathbb{S}}

\newcommand{\CP}{{\mathbf{CP}}}

\newcommand{\RP}{\mathbf{RP}}

\newcommand{\Map}{\mathrm{Map}}

\newcommand{\pullback}{\arrow[dr, phantom, "\lrcorner", very near start]}

\newcommand{\fZ}{\mathfrak{Z}}

\newcommand{\Alg}{\mathrm{Alg}}

\newcommand{\Eoo}{{\mathbf{E}_\infty}}

\newcommand{\CAlg}{\mathrm{CAlg}}

\newcommand{\Perf}{\mathrm{Perf}}

\newcommand{\E}[1]{{\mathbf{E}_{{#1}}}}

\newcommand{\GL}{\mathrm{GL}}

\newcommand{\KU}{\mathrm{KU}}

\newcommand{\ku}{\mathrm{ku}}

\newcommand{\Conf}{\mathrm{Conf}}

\newcommand{\B}{\mathrm{B}}

\newcommand{\W}{\mathrm{W}}

\newcommand{\cyc}{\mathrm{cyc}}

\newcommand{\one}{\mathbbm{1}}

\renewcommand{\log}{\mathrm{log}}

\newcommand{\cS}{\mathcal{S}}

\renewcommand{\S}{Section }

\newcommand{\NN}{\mathbf{N}}

\makeatletter
\providecommand{\leftsquigarrow}{%
  \mathrel{\mathpalette\reflect@squig\relax}%
}
\newcommand{\reflect@squig}[2]{%
  \reflectbox{$\m@th#1\rightsquigarrow$}%
}
\makeatother

\newcommand{\pw}[1]{[\![#1]\!]}

\newcommand{\EE}{\mathbf{E}}

\usepackage{slashed}

\usepackage{relsize}
\usepackage[bbgreekl]{mathbbol}
\usepackage{amsfonts}
\DeclareSymbolFontAlphabet{\mathbb}{AMSb} 
\DeclareSymbolFontAlphabet{\mathbbl}{bbold}

\newcommand{\Ainf}{A_\mathrm{inf}}

\usepackage{marginnote}

\title{Ramification in $\E{n}$-algebras}

\author{Sanath K. Devalapurkar}
\address{Fuld Hall, Princeton NJ 08540}
\email{sdevalapurkar@gmail.com}

\author{Ishan Levy}
\address{Fuld Hall, Princeton NJ 08540}
\email{ishanl@ias.edu}

\begin{document}

\maketitle

{\small \tableofcontents}

\section{Introduction}
Throughout this paper we work locally at a prime $p$.
Inspired by the work of Schw\"anzl-Vogt-Waldhausen in \cite{adjoining-roots}, the first author showed:
\begin{theorem}[{\cite[Theorem 1.3]{roots-of-unity}}]
	\label{thm:einfty}
	Let $h\geq 1$ and $p^k>2$. Then a $T(h)$-local\footnote{The version of this result for $K(h)$ and $T(h)$ are equivalent by the nilpotence theorem: see \cite[Lemma 2.3]{LMMT}.} $\EE_\infty$-ring cannot have a primitive $p^k$th 
    root of unity $\zeta_{p^k}$ in its $\pi_0$.
\end{theorem}

This result can be viewed as saying that ramification in the homotopy rings of $\EE_{\infty}$-rings behaves quite differently at heights $\geq 1$ than at height $0$. Our goal in this note is to understand the extent to which this phenomenon continues for $\EE_n$-rings for $n<\infty$. 

Given an even $\EE_2$-ring $R$, it is not hard to adjoin roots of unity as an 
$\EE_1$-ring. Namely, one starts with the polynomial ring $R[x]$, and use 
\cite[Corollary 3.2]{angeltveit2008topological} to obtain an $\EE_1$-ring structure 
on $R[x]/\Phi_m(x)$, where $\Phi_m$ is the $m^{\text{th}}$ cyclotomic polynomial. 
Thus the first $m$ for which there might be obstructions to adding roots of unity 
to $\EE_n$-rings is $n=2$. We show that there are indeed obstructions in this case.

\begin{customthm}{A}
	\label{thm:roote2}
	Let $h\geq 1$. There is a $k\geq 0$ depending on $p,h$ such that a $T(h)$-local $\EE_2$-ring cannot have $\zeta_{p^k}$ in $\pi_0$.
\end{customthm}

We obtain this result as a consequence of the more general result below, which obstructs perfectoid rings from mapping into the $\pi_0$ of nonzero $T(h)$-local $\EE_2$-rings. \Cref{thm:roote2} is a consequence of taking $S$ to be $\Z_p^{\cyc}$ the $p$-adic completion of the maximal $p$-cyclotomic extension of the integers:

\begin{customthm}{B}
	\label{thm:perfectoidpi0}
	Let $R$ be an $\EE_2$-ring with a map $S \to \pi_0R$ where $S$ is a perfectoid ring. Then there is a map of $\EE_2$-rings $S \to R$. In particular $R$ is an $\EE_1$-$S$-algebra.
\end{customthm}

Given \Cref{thm:roote2}, we can make the following definition:

\begin{definition}
	Fix a prime $p$, a height $h >0$, and an $n\geq 2$. The numbers $r_{p,h,n}$ to be the largest natural number such that there exists a $T(h)$-local $\EE_n$-algebra with a $p^{r_{p,h,n}}$-th root of unity in $\pi_0$.
\end{definition}

The numbers $r_{p,h,n}$ are nonincreasing in $n$, and measure the extent to which roots of unity must not exist in $T(h)$-local $\EE_n$-algebras. We prove the following results about the numbers $r_{p,h,n}$:

\begin{customthm}{C}\label{thm:optimalroots}
	The following are true about the $r_{p,h,n}$:
	\begin{itemize}
		\item $r_{p,1,4}$ is $0$ for $p>2$ and $1$ for $p=2$.
		\item Fix $p$ and $h$. Then for $n\gg 0$, $r_{p,h,n}$ is $0$ for $p>2$ and $1$ for $p=2$.
	\end{itemize}
\end{customthm}

At height $h=1$, the above result is proven by using the existence of $\delta$-operations $\pi_0(A) \to \pi_0(A/p^2)$. In fact, our proof also shows the more general obstruction to ramification in $\EE_4$-$\KU$-algebras:
\begin{customthm}{D}
	Let $A$ be a $K(1)$-local $\EE_4$-$\KU$-algebra. If there is a map from $\cO_K \to \pi_0A$ where $\cO_K$ is the ring of integers of a ramified extension $K$ of $\QQ_p$, then $A = 0$.
\end{customthm}

We also study in \Cref{section:strict} obstructions to adjoining roots of unity in a strict way, for example showing in \Cref{cor:e1a2} that cyclotomic polynomial $\Phi_{p^k}(q) $ is not $\EE_3$-strict in $\KU[q]^{\wedge}_p$ for $p^k>2$.
This is an obstruction to constructing $\KU[\zeta_{p^k}]$ as an $\EE_2$-algebra in a strict way. 

Despite these obstructions, we were suprised to learn that there are $K(1)$-local $\EE_2$-algebras $A$ with a $p$th root of unity.
\begin{example}\label{exm:angeltveit}
	It follows from \cite[Theorem 5.17]{angeltveit2008topological} that there is a $K(1)$-local $\EE_2$-$\KU$-algebra with $\pi_0$ isomorphic to $\Z_p[\zeta_p]$. 
	This can be obtained from a Hochschild cohomology of a $2$-periodic form of $K(1)$. We note \cite[Conjecture 5.20]{angeltveit2008topological} describes an analog of this at higher height.
\end{example}

The above example shows that the number $r_{p,h,n}$ are interesting, and we believe it would be interesting to compute them.


\begin{remark}\label{rmk:higher roots}
	One can attempt to study analogues of the statements discussed here for higher roots 
	of unity in the sense of \cite[Definition 4.2]{chromatic-cyclotomic}.
	In this case, recall that a primitive $p^r$th root of unity of height $n$ for some $r\geq 1$ in an 
	commutative algebra $R$ in a presentably symmetric monoidal stable category 
	 is an $\EE_{\infty}$-map $\one[B^n\ZZ/p^r] \to R$ such that the pushout of commutative algebras 
	of the span $\one[B^n\ZZ/p^{r-1}] \leftarrow \one[B^n\ZZ/p^r] \to R$ is $0$.

	We then claim that higher primitive roots of unity don't exist at the wrong height, i.e., 
	if $R$ is a $T(h)$-local $\EE_{\infty}$-algebra with $h \neq n$ and a primitive $p^r$th root of unity, 
	then $R = 0$. If $n>h$, then $B^n\ZZ/p^r$ is $T(h)$-acyclic, so the pushout is $R$ itself.
	If $n<h$, then because the $T(n)$-local category is virtually $(\FF_p,n)$-orientable \cite[Proposition 7.27]{chromatic-fourier}, we can add a 
	 height $n$ $p^r$th root of unity, we can identify this pushout with its Fourier dual \cite[Theorem 7.33]{chromatic-fourier}, which is
	$\one^{B^{h-n}((\ZZ/p^{r-1})^{\vee})} \leftarrow \one^{B^{h-n}((\ZZ/p^r)^{\vee})} \to R$. However, for $h-n\geq 1$, tensoring along the map $\one^{B^{h-n}((\ZZ/p^r)^{\vee}))} \to \one^{B^{h-n}((\ZZ/p^{r-1})^{\vee})}$ 
	is conservative, since the target ring has a $B^{h-n-1}(\ZZ/p)$-action whose fixed points is the source, and fixed points commutes with colimits by $\infty$-semiadditivity \cite{tn-ambi}.

	We note that the notion of higher primitive roots of unity considered here is much stronger than what is considered in this paper, 
	it is defined for $\EE_{\infty}$-algebras and is a strict notion (see \Cref{section:strict}).
	It would be interesting to explore the extent to which this phenomenon persists for weaker notions of higher roots of unity.
	
\end{remark}

\subsection*{Acknowledgments}
We are grateful to Arpon Raksit and Tomer Schlank for conversations related to this work. We thank John Rognes for pointing us to Angeltveit's paper for \Cref{exm:angeltveit}, and thank Tomer Schlank for pointing us to \cref{rmk:higher roots}.

Most of the work for this paper was done around 2022. At the time, the first author was 
supported by the PD Soros Fellowship and the NSF GRFP. The second author was supported 
by the NSF GRFP under Grant No. 1745302. The second author is currently supported by the Clay Research Fellowship and NSF Grant No. DMS-2424441.

\section{Perfectoid rings in $\pi_0$}

In this section, we prove our main result:

\begin{customthm}{B}
	Let $R$ be an $\EE_2$-ring with a map $S \to \pi_0R$ where $S$ is a perfectoid ring. Then there is a map of $\EE_2$-rings $S \to R$. In particular, $R$ is an $\EE_1$-$S$-algebra.
\end{customthm}

As a consequence, we deduce \Cref{thm:roote2}.

\begin{proof}[Proof of \Cref{thm:roote2}]
	Suppose that the conclusion of \Cref{thm:roote2} is false. Then for each 
	$k\geq 0$, there exists a $T(h)$-local $\EE_2$-ring $R_k$ such that $\pi_0 R_k$ 
	contains a primitive $p^k$th root of unity $\zeta_{p^k}$. Let $R$ be an ultraproduct 
	of the $R_k$. Then $R$ has nonzero $T(h)$-localization: this is because it is a 
	filtered colimit of non-zero $T(h)$-local rings given by products of the $R_k$, and 
	$T(h)$-local filtered colimits of nonzero $T(h)$-local rings are nonzero (see the 
	proof of \cite[Lemma A.11]{chromatic-nullstellensatz}). But $\pi_0(R)$ contains a 
	system of primitive $p$th roots of unity, so there is a map 
	$\Z_p^\mathrm{cyc} \to \pi_0(R)$. Applying \Cref{thm:perfectoidpi0} to the perfectoid 
	ring $\Z_p^\mathrm{cyc}$, we obtain a contradiction, since $\Z_p^\mathrm{cyc}$ is 
	$T(h)$-acyclic.
\end{proof}

We now turn to proving \Cref{thm:perfectoidpi0}. Recall that there is a spherical Witt vector functor $\bW:\Perf_{\FF_p} \to \CAlg(\Sp_p)$. To prove \Cref{thm:perfectoidpi0}, we first show that $R$ lifts to an $\EE_2$-algebra over $\bW((S/p)^\flat)$, which follows from the fact that spherical Witt vectors always act centrally on $\EE_n$-algebras (\Cref{prop:centrality} below). Then we conclude by applying a perfectoid version of the Hopkins--Mahowald theorem due to Mao (\Cref{thom-spectra}), which implies that the existence of an $\EE_2$-map from $S$ to an $\EE_2$-$\bW((S/p)^\flat)$-algebra can be checked on $\pi_0$.

We begin with the proposition below, which is an analogue of \cite[Proposition 2.2]{chromatic-nullstellensatz} in the setting of $\EE_n$-algebras.
\begin{prop}\label{lem:wittenetale}
	For $n\geq2$, the composite $\bW:\Perf_{\FF_p} \to \CAlg(\Sp_p) \to \Alg_{\EE_n}(\Sp_p)$ is fully faithful with right adjoint $(-)^{\flat}$, sending $R$ to $(\pi_0R/p)^{\flat}$.
\end{prop}
\begin{proof}
	We first show that for $n\geq 2$, $\bW(A)$ is formally \'etale as a $p$-complete $\EE_n$-algebra, i.e that the $\EE_n$-cotangent complex vanishes. To do this, it suffices by \cite[Theorem 2.26, Proposition 3.16]{francis2013tangent} to show that $\int_{\RR^n-0}\bW(A) \cong \bW(A)$. Because $\int_{\RR^n-0}\bW(A)\cong \bW(A)\otimes_{\int_{\RR^{n-1}-0}\bW(A)}\bW(A)$, it suffices to show this for $n=2$, and since $\bW$ is symmetric monoidal, it suffices to show this in $\Perf_{\FF_p}$ for $A$. This follows since
	$$\int_{\RR^2-0}A \cong A\otimes_{A\otimes A}A \cong A$$
	
	It follows from the existence of the adjunction that $\bW(A)$ is fully faithful, so it only remains to show the adjunction, i.e produce a natural isomorphism $$\Map_{\Alg_{\EE_n}(\Sp_p)}(\bW(A),R) \cong \Map_{\Perf_{\FF_p}}(A,(\pi_0R/p)^{\flat})$$ To do this, since $\bW(A)$ is connective, it suffices to assume that $R$ is connective. Then since the Postnikov tower of $R$ is a tower of square zero extensions and $\bW(A)$ is formally \'etale, we reduce to the case $R$ is discrete, where the result is now clear, since $\pi_0\bW(A) \cong \W(A)$.
\end{proof}

The next key observation in proving \Cref{thm:perfectoidpi0} is that the counit map $\bW(R^{\flat}) \to R$ is always central. Recall from \cite[Section 5.3]{HA} and \cite[Definition 3.1, Proposition 3.16]{francis2013tangent} that the $\EE_{n+1}$-center of an $\EE_n$-algebra $R$ is given by the endomorphism algebra $\End_{\int_{S^{n-1} \times \RR} R}(R)$.

\begin{prop}\label{prop:centrality}
	Given a $p$-complete $\EE_n$-algebra $R$ with $n\geq2$, $R$ is an $\EE_n$-$\bW(R^{\flat})$-algebra. Moreover, $\fZ_{\EE_n}(R)^{\flat} \cong R^{\flat}$, where $\fZ_{\EE_n}(R)$ is the $\EE_{n+1}$ center of $R$.
\end{prop}

\begin{proof}
	First we refine the $\EE_n$-algebra map $\bW(R^{\flat}) \to R$ to an $\EE_{n}$-algebra map $\bW(R^{\flat}) \to \fZ_{n}(R)$. To do this, it suffices by the universal property of $\fZ_{n}(R)$ (\cite[Definition 5.3.1.2]{HA}) to produce a diagram of $\EE_n$-algebras
	
		\begin{center}
		\begin{tikzcd}
			R\ar[rr,equal]\ar[dr,"1\otimes R",swap] & &R. \\
			& \ar[ur]\bW(R^{\flat})\otimes R & 
		\end{tikzcd}
	\end{center}

	To do this, we consider the base-change along the multiplication map, 
	which a priori is a map of $\EE_{n-1}$-algebras: 
	$$m:\bW(R^{\flat}) \otimes R \to (\bW(R^{\flat})\otimes R)\otimes_{\bW(R^{\flat}\otimes R^{\flat})}\bW(R^{\flat}).$$ Next, we note that the target of $m$ is an idempotent algebra over the source: this follows because $\bW$ takes surjective maps to idempotent algebras, and we are basechanging along such a map. Thus $m$ canonically refines to an $\EE_n$-algebra map. 
	Because $1\otimes R$ can be written as an $\EE_{n-1}$-algebra map as $R \to R\otimes_{\bW(R^{\flat})}\bW(R^{\flat}\otimes R^{\flat})$ via the right unit, we find that $m \circ (1\otimes R)$ is an equivalence\footnote{A subtlety is that it apriori is only the \textit{identity} map as an $\EE_{n-1}$-ring map. But any equivalence is equivalent to the identity map, so this doesn't matter.}. This gives the desired triangle.
	
	Since we have an $\EE_n$-algebra map $\bW(R^{\flat}) \to \fZ_n(R)$, by applying 
	\Cref{lem:wittenetale} it follows that the map $\bW(R^{\flat}) \to \fZ_n(R)$ 
	automatically refines to an $\EE_{n+1}$-map, since the mapping spaces in $\EE_n$ 
	and $\EE_{n+1}$-algebras are the same. Since $R$ is an $\EE_n$-$\fZ_n(R)$-algebra, it follows that it is also an $\EE_n$-$\bW(R^{\flat})$-algebra.
	
	To finish, we now show that $\fZ_n(R)^{\flat} \to R^{\flat}$ is an isomorphism. Surjectivity follows from the lift $\bW(R^{\flat}) \to \fZ_n(R)$ we have already produced. For injectivity, if $a$ is in the kernel, then we consider the map of $\EE_{n+1}$-algebras $g:\fZ_n(R) \to \fZ_n(R)\otimes_{\bW(\FF_p[x]^{\sharp})}\bW(\FF_p)$ where $x$ is sent to $a$. Here $\sharp$ denotes the colimit perfection. Since $\bW(\FF_p)$ is an idempotent algebra over $\bW(\FF_p[x]^{\sharp})$, the target is an idempotent algebra over the source. Since $a$ is in the kernel, we learn that $R^{\flat}$ is an $\EE_n$-$\fZ_n(R)\otimes_{\bW(\FF_p[x]^{\sharp})}\bW(\FF_p)$-algebra, so that $g$ has a section via the universal property of the center. It follows that $a = 0$. 
%
\end{proof}

\begin{remark}
	The lift of $R$ to an $\EE_n$-$\bW(R^{\flat})$-algebra is essentially unique. Indeed, the space $L$ of lifts is the pullback
	\begin{center}
		\begin{tikzcd}
			L\pullback\ar[r]\ar[d] & \Map_{\Alg_{\EE_{n+1}}(\Sp_p)}(\bW(R^{\flat}),\fZ_{\EE_n}(R))\ar[d]\\
			*\ar[r] & \Map_{\Alg_{\EE_n}(\Sp_p)}(\bW(R^{\flat}),R)
		\end{tikzcd}
	\end{center}
	which is contractible since the right vertical map is an isomorphism as $\fZ_{n}(R)^{\flat} \cong R^{\flat}$.
\end{remark}

	The last ingredient for \Cref{thm:perfectoidpi0} is the perfectoid version of the Hopkins--Mahowald theorem, which we now explain. Let $S$ be a perfectoid ring. Then the map $\W(S^{\flat})\to S$ on $\pi_0$ is Fontaine's map $\theta: \Ainf(S) \to S$. The kernel of $\theta$ is a principal ideal, with a generator $\xi$, giving $S = \Ainf(S)/\xi$ as discrete rings. This has the following lift to $\EE_2$-spectral algebraic geometry:
	
\begin{prop}[{\cite[Theorem 1.13]{zhouhang-mao}}]\label{thom-spectra}
Let $S$ be a perfectoid ring. Then there is an isomorphism of $\EE_2$-algebras $\bW(S^{\flat})\sslash_{\EE_2}\xi \cong S$, where $\sslash_{\EE_2}$ is the $\EE_2$-quotient, i.e., the free $\EE_2$-$\bW(S^{\flat})$-algebra equipped with a nullhomotopy of $\xi$.
\end{prop}
	For $S = \FF_p$, this is the Hopkins--Mahowald theorem \cite[Theorem 5.1]{antolin2019simple}.

\begin{proof}[Proof of \Cref{thm:perfectoidpi0}]
	Because we have a map $\pi_0S \to \pi_0R$, we have a map $\bW(S) \to \bW(R)$, 
	making $R$ an $\EE_2$-$\bW(S)$-algebra. Since $d$ is sent to $0$ in $R$,  \Cref{thom-spectra} produces an $\EE_2$-algebra map $S \to R$.
\end{proof}

\section{Strict elements}\label{section:strict}

We now show that there are strong restrictions on $\EE_n$-algebras with \textit{strict} roots of unity for $n>1$. We first define the notion of an $\EE_m$-strict element of an $\EE_n$-algebra, which is inductive on $m-n$.
\begin{definition}
	Let $R$ be an $\EE_n$-algebra, $x \in \pi_0R$, and $m \in \NN\cup \infty$. We say $x$ is an $\EE_m$-\textit{strict
	element} if
\begin{itemize}
	\item If $m \leq n$, we ask for the existence of an $\EE_m$-algebra map $\SP[t] \to R$ sending the generator $t$ to $x$.
	\item If $n<m<\infty$, we ask that $x$ lift to an $\EE_m$-strict element of $\pi_0\fZ_{n}R$.
	\item If $m = \infty$, we ask for $x$ to be $\EE_k$-strict for all $k <\infty$ compatibly in $k$.
\end{itemize}
\end{definition}

For example, \Cref{prop:centrality} can be interpreted as saying that when $x$ lifts to $\W(R^{\flat})$, it is $\EE_{\infty}$-strict. If $x$ is $\EE_m$-strict, then so are all of the powers of $x$ (including negative powers, if they exist). There is also a variation of the notion of $\EE_m$-strict, where for an operad $\cO \to \EE_n$, we say that an element $x$ in $\pi_0$ of an $\EE_n$-algebra $R$ is $\cO$-strict if there is an $\cO$-algebra map $\SP[t] \to R$ sending the generator $t$ to $x$.

\begin{prop}\label{power-series-algebra}
    Suppose that $p^k \neq 1,2$.
	Let $R$ be a $p$-complete $\EE_2$-ring with a primitive $p^k$th root of unity $\zeta_{p^k}$ which is $\EE_3$-strict. Then $R$ admits an $\EE_2$-map from $\Z_p[\zeta_{p^k}]$, and in particular is an $\EE_1$-$\Z_p[\zeta_{p^k}]$-algebra.
\end{prop}

\begin{proof}
	Since $\zeta_{p^k}$ is an $\EE_3$-strict element, there is an $\EE_3$-map $\SP[q] \to \fZ_{\EE_2}R$ which makes $R$ into an $\EE_2$-$\SP[q]$-algebra.
	The claim follows once we show that the free $\EE_2$-$\SP_p[q]$-algebra with $q$ a root of unity in $\pi_0$ is $\Z_p[\zeta_{p^k}]$. This is \cite[Theorem 1.19]{zhouhang-mao}, but we quickly recall the argument for completeness. There is an $\EE_2$-map $\SP_p[q]\sslash_{\EE_2}\Phi_{p^k}(q) \to \Z_p[\zeta_{p^{n+1}}]$. Both target and source are complete with respect to $q-1$, so it suffices to check that this map is an equivalence after base-changing along the map $\SP_p[q] \to \SP_p$ sending $q\mapsto 1$. After base-change, this becomes the map $\SP_p\sslash_{\EE_2}p \to \FF_p$, which is an equivalence by the Hopkins--Mahowald theorem.
\end{proof}


%

If $R$ is an even $\EE_2$-algebra, there are no obstructions to any $x \in \pi_0R$ being $\EE_2$-strict, since the cells of $\SP[t]$ as an $\EE_2$-ring are in even degrees. As a consequence we can adjoin any algebraic equation\footnote{One can also add multiple generators and relations at the same time.} $R[x]/(f)$ as a basechange $R[x]\otimes_{\SP_p[f]} \SP$, in particular allowing us to add roots of unity as an $\EE_1$-algebra (see also \cite[Theorem 3.1]{angeltveit2008topological}).

Our next goal is to show using the Rezk logarithm \cite{rezklog} that when $R$ is a $K(1)$-local $\EE_{\infty}$-ring, any unit which is $\EE_2\otimes \AA_2$-strict (for example an $\EE_3$-strict element) is $\EE_{\infty}$-strict. Recall that the $\pi_0$ of $K(1)$-local $\EE_{\infty}$-rings naturally admits the structure of a $\delta$-ring, with lift of Frobenius $\psi(x) = x^p + p\delta(x)$, and that $L_{K(1)}\SP\{t\}$ is the free polynomial algebra on $\delta^i(t)$ for $i\geq 0$ \cite{hopkins2014k}.

We first record a well known lemma. Let $\SP\{t\}$ denote the free $\EE_{\infty}$-ring on a class $t$ in degree $0$.
\begin{lemma}\label{lem:k1localstrict}
	There is an equivalence of $\EE_\infty$-rings $L_{K(1)}\SP[t]\cong L_{K(1)}\SP\{t\}\sslash_{\EE_{\infty}}\delta(t)$. Here, the quotient on the right-hand side is taken in $K(1)$-local $\EE_{\infty}$-rings.
\end{lemma}

\begin{proof}
	Since $\delta(t) = 0$ in $L_{K(1)}\SP[t]$, there is an $\EE_\infty$-map 
	$L_{K(1)}\SP\{t\}\sslash_{\EE_{\infty}}\delta(t) \to L_{K(1)}\SP[t]$. To see this 
	is an equivalence, McClure's theorem \cite[Chapter IX]{bmms} (see also \cite{hopkins2014k}) shows that $\pi_*L_{K(1)}\SP\{t\}$ is a $p$-completed polynomial ring over $\pi_*L_{K(1)}\SP$ on $\delta^i(t)$ for $i\geq 0$. 
	Since $L_{K(1)}\SP\{t\}\sslash_{\EE_{\infty}}\delta(t) \cong L_{K(1)}\SP\{t\} \hat{\otimes}_{L_{K(1)}\SP\{\delta(t)\}} L_{K(1)}\SP$, it follows from the $p$-completed bar spectral sequence that the map is an equivalence.
\end{proof}

\begin{prop}\label{log}
Let $A$ be a $K(1)$-local $\Eoo$-$\KU_p$-algebra with $\pi_0(A)$ torsion free, and let $\xi \in \pi_0(A)^{\times}$. If $\xi$ is $\EE_2\otimes \AA_2$-strict, then it is $\EE_{\infty}$-strict. In particular, $\delta(\xi) = 0$.
\end{prop}
\begin{proof} 
	By \Cref{lem:k1localstrict}, it suffices to show that $\delta(\xi) = 0$ if $\xi$ is $\EE_2\otimes \AA_2$-strict.
	We begin by observing that $$\Map_{\Alg_{\EE_2\otimes \AA_1}(\Sp)}(\SP[t],A) \cong \Map_{\Alg_{\EE_2\otimes \AA_2}(\cS)}(\NN,\Omega^{\infty}A)$$ where the target has the $\EE_{\infty}$ structure coming from multiplication. Since $\xi$ is in a unit, using $(-)_{\xi}$ to denote the component sending $1$ to $\xi$, we get
	$$\Map_{\Alg_{\EE_2\otimes \AA_1}(\Sp)}(\SP[t],A)_{\xi} \cong \Map_{\Alg_{\EE_2\otimes \AA_2}(\cS)}(\Z,\GL_1A)_{\xi} \cong \Map_{\Alg_{\AA_2}(\cS)}(B^2\Z,B^2\GL_1(A))_{\xi}$$ 
In other words, we need to show that given a commutative diagram of the form
$$\xymatrix{
\CP^\infty \times \CP^\infty \ar[d]^-\mu \ar[r]^-{\xi,\xi} & \B^2\GL_1(A) \times \B^2\GL_1(A) \ar[d] \\
\CP^\infty \ar[r]_-{\xi} & \B^2\GL_1(A).
}$$
$\delta(\xi) = 0$. 

To understand this, we will use the $K(1)$-local logarithm, which defines an infinite loop map $\ell: \B^2\GL_1(A) \to \Omega^\infty \Sigma^2 A$ since $A$ is $K(1)$-local. After composing with the logarithm, we obtain a diagram
$$\xymatrix{
\Sigma^\infty_+(\CP^\infty \times \CP^\infty) \ar[d]^-\mu \ar[r]_-{\ell(\xi), \ell(\xi)} & \Sigma^2(A \oplus A) \ar[d] \ar[r]^-\sim & A \oplus A \ar[d] \\
\Sigma^\infty_+ \CP^\infty \ar[r]_-{\ell( \xi)} & \Sigma^2 A \ar[r]_-\sim & A.
}$$
The final column uses multiplication by $\beta$ to identify $\Sigma^2 A \simeq A$. 
Since $A$ is complex orientable and $2$-periodic, a choice of orientation and periodicity generator gives an isomorphism $[\Sigma^{\infty}_+\CP^{\infty},L_{K(1)}A] \cong \pi_0L_{K(1)}A\pw{x}$, to which $\ell(\xi)$ corresponds to a power series $f$. The commutativity of the diagram implies that $f$ is a homomorphism from the formal group of $A$ to the additive formal group.

Over a torsion free ring, up to rescaling the Lie algebra, there is a unique homomorphism between any two formal groups, which to the additive formal group is given by the logarithm. In the case of $A$, which is a $\KU_p$-algebra, this is given by
%
$f(x) = a\sum_{n\geq 1} \frac{x^n}{n}$. Since $\pi_0A$ is $p$-torsion free, and $A$ is $p$-complete, $\pi_0A$ is $p$-adically separated, so $a = 0$ for $f$ to be an integral power series, i.e $\ell(\xi) = 0$.

By \cite[Theorem 1.9]{rezklog}, we have
$$0 = \ell(\xi) = \sum_{n\geq 1} (-1)^n \frac{p^{n-1}}{n} \left(\frac{\delta(\xi)}{\xi^p}\right)^n = \frac{1}{p} \log\left(\frac{\xi^p}{\psi(\xi)}\right).$$
Using that $p\in \pi_0(A)$ has divided powers, we can multiply by $p$ and exponentiate to get
$$1 = \exp\left(p\ell(\xi)\right)
= \sum_{n\geq 0} \ell(\xi)^n \frac{p^n}{n!} = \frac{\xi^p}{\psi(\xi)} $$
Thus $\psi(\xi) = \xi^p+p\delta(\xi) = \xi^p$ so $\delta(\xi) = 0$.
\end{proof}

As discussed earlier, we can adjoin primitive $p^k$th roots of unity to even algebras (such as $\KU$) as an $\EE_1$-ring by using the fact that the element $\Phi_{p^k}(q)$ is $\EE_2$-strict in $\KU[q]$. The following corollary shows that this method cannot be improved even to get an $\EE_1\otimes \AA_2$-structure, and nor can any other similar kind of strict quotient.
\begin{corollary}\label{cor:e1a2}
	The element $\Phi_{p^k}(q)$ is not $\EE_2\otimes \AA_2$-strict in $\KU[q]^{\wedge}_p$ if $p^k \neq 1,2$. 
\end{corollary}
\begin{proof}
	If $\Phi_{p^k}(q)$ were $\EE_2\otimes \AA_2$-strict, it would be $\EE_{\infty}$-strict by \Cref{log}. But then the tensor product $\SP\otimes_{\SP[t]}\KU[q]^\wedge_p$ would be a $K(1)$-local $\EE_{\infty}$-ring having a primitive $p^k$th root of unity in $\pi_0$, contradicting \Cref{thm:einfty}.
\end{proof}

\section{$K(1)$-local $\E{4}$-$\KU$-algebras}

The goal of this section is to prove \Cref{thm:optimalroots}. The key point 
is \Cref{prop:enfroblift} below, that $K(1)$-local $\EE_{2n+2}$-algebras have 
a lift of Frobenius mod $p^n$. This makes the obstruction to roots of unity 
from \Cref{thm:einfty} still hold, and it moreover generalizes to more general ramified extensions. We then use an ultraproduct argument to 
show that $r_{p,h,n}$ agrees with the number for $\EE_{\infty}$ for $n$ sufficiently large.

\begin{prop}\label{prop:confn}
	There is a square as below, where the vertical maps are isomorphisms, and the lower horizontal map is the $p^n$-Bockstein:	
	\begin{center}
		\begin{tikzcd}
			\mathrm{Conf}_p(\RR^{2n+1})\otimes \KU\ar[r]\ar[d] &B\Sigma_{p}\otimes \KU\ \ar[d,"\mathrm{tr}"]\\
			\Sigma^{-1}\KU/p^n\ar[r,"\beta_{p^n}"] & \KU.
		\end{tikzcd}
	\end{center}
\end{prop}

\begin{proof}
    For $p>2$, this follows from the commutative diagram on \cite[Section 6.2, Page 44]{brantner-hahn-knudsen}. For $p=2$, note that $\mathrm{Conf}_2(\RR^{2n+1}) \cong \RP^{2n}$, and the map to $B\Sigma_2 = \RP^\infty$ is the standard inclusion. Upon $2$-completion, both $\KU \otimes \RP^{2n}$ and $\KU \otimes \RP^\infty$ are perfect $\KU_2$-modules, so it suffices to identify the $\KU_2$-linear dual of the map $\KU_2 \otimes \RP^{2n} \to (\KU \otimes \RP^\infty)_2$ with the map $\KU_2 \to \KU/2^n$. This follows from \cite[Corollary 4.17]{powell-notes}.
\end{proof}

Below, we use the map $p:X/p^n \to X/p^{n+1}$ for a spectrum $X$.
\begin{prop}\label{prop:enfroblift}
	Suppose that $R$ is a $K(1)$-local $\EE_{2n+2}$-$\KU$-algebra. Then there is a natural map
	$$\delta:\pi_0R \to \pi_0(R/p^{n})$$
	such that
	\begin{align*}
	   \delta(x+y) & = \delta(x)+\delta(y) - \sum_{i=1}^{p-1}\frac{\binom{p}{i}}{p}x^iy^{p-i}, \\
	   \delta(xy) & = \delta(x)y^p+\delta(y)x^p+p\delta(x)\delta(y).
	\end{align*}
	In particular, the map $\phi:\pi_0R \to \pi_0(R/p^{n+1})$ defined by $\phi(x) = x^p+p\delta(x)$, is a ring homomorphism lifting the Frobenius map.
\end{prop}

\begin{proof}
	We will define the map $\delta$ for all $K(1)$-local $\EE_{2n+1}$-$\KU$-algebras, and 
	show that it satisfies the desired relations for all $\EE_{2n+2}$-algebras.
	
	To define the map $\delta$, it suffices to do this in the universal example of the 
	free $\EE_{2n+1}$-$\KU$-algebra on a class $x$ in degree $0$. 
	From \Cref{prop:confn}, we learn that in arity $p$, the map from the free 
	$\EE_{2n+1}$-algebra to the free $\EE_{\infty}$-algebra is an isomorphism 
	on $\pi_0/p^n$, allowing us to define $\delta$ as the unique lift of the 
	associated $\EE_{\infty}$-operation.
	
	For the relations, it suffices to prove them in the universal 
	$\EE_{2n+2}$-algebra on two classes in degree $0$. The relation for $\delta(x+y)$ is witnessed in the arity $p$ part of this free algebra, which is $((\KU\oplus \KU)^{\otimes p}\otimes (\Conf_p(\RR^{2n+2}))_+))_{h\Sigma_p}$. We claim that the map
	$$((\KU\oplus \KU)^{\otimes p}\otimes (\Conf_p(\RR^{2n+2}))_+))_{h\Sigma_p}/p^n \to ((\KU\oplus \KU)^{\otimes p}_{h\Sigma_p}/p^n$$
	is injective on the image of $\pi_0((\KU\oplus \KU)^{\otimes p}\otimes (\Conf_p(\RR^{2n+1}))_+))_{h\Sigma_p}/p^n$ in $\pi_0$. The result then follows, since the target is the arity $p$ part of the universal $\EE_{\infty}$-algebra on two generators, where the desired relation holds.
	
	To see the claim, we note that ($p$-adically), $((\KU\oplus \KU)^{\otimes p}\otimes (\Conf_p(\RR^{i}))_+))_{h\Sigma_p}$ naturally is a sum of two copies of $\KU^{\otimes p}\otimes \Conf_p(\RR^i)_{h\Sigma_p}$, and a sum of copies of $\Conf_p^{\mathrm{ord}}(\RR^i)_+$ indexed by the nontrivial $\Sigma_p$ orbits of $2^p$. The claim then follows from the fact that the inclusion $\Conf_p^{\mathrm{ord}}(\RR^i) \to \Conf_p^{\mathrm{ord}}(\RR^{i+1})$ is null, so the image factors through a section from $\Conf_p^{\mathrm{ord}}(\RR^\infty)$.
	
	The proof of the relation for $\delta(xy)$ is nearly identical, except the relation takes place in arity $2p$ instead of $p$.
	
	Once these relations are proven, it follows that $\varphi$ is a ring map and 
	a lift of Frobenius, for the same reason that a delta structure gives a lift of Frobenius.
\end{proof}

\begin{remark}\label{rmk:browderbracket}
	The relations satisfied by $\delta$ are different if $R$ is only an $\EE_{2n+1}$-$\KU$-algebra, and can involve Browder brackets.

	We sketch how this works for the relation $\delta(xy) = \delta(x)y^2+\delta(y)x^2+2\delta(x)\delta(y)$ for $p=2$ and $n=1$. In this case, $\beta\delta$ lifts to a natural map $\pi_0(R) \to \pi_{2}(R/2)$ for an $\EE_3$-$\ku$-algebra.
	Moreover, mod $\beta$, it is easy to see that this operation is the Dyer--Lashof operation $Q_2$ on $\EE_3$-$\FF_2$-algebras. The nontriviality of the Browder bracket term then 
	follows from the fact that the Cartan formula for $Q_2$ of an $\EE_3$-algebra has a 
	Browder bracket term: it is given by $Q_2(xy) = Q_2(x)y^2+x^2Q_2(y)+x[x,y]y$ (see 
	\cite[Theorem 5.2]{lawson-dl}), using that $Q_1$ vanishes since the classes $x,y$ are 
	integral.

	It is possible that arguments similar to ours are able to prove the analog of \Cref{thm:ramificationobstruction} for $K(1)$-local $\EE_3$-$\KU$-algebras, 
	and we believe it would be interesting to work this out.
\end{remark}

Using \Cref{prop:enfroblift}, we can prove \Cref{thm:ramificationobstruction}, which is the key input to \Cref{thm:optimalroots}.

\begin{customthm}{D}\label{thm:ramificationobstruction}
	Let $A$ be a $K(1)$-local $\EE_4$-$\KU$-algebra. If there is a map from $\cO_K \to \pi_0A$ where $\cO_K$ is the ring of integers of a ramified extension $K$ of $\QQ_p$, then $A = 0$.
\end{customthm}

\begin{proof}
	Let $\pi$ be a uniformizer of $\cO_K$, with valuation $\frac 1 e$ where $e>1$ by assumption. 
	Let $a$ be an integer minimally chosen so that $\frac {ap} e\geq 1$. Note that $a<e$. Then $\varphi(\pi^a) \equiv \pi^{ap} \equiv 0 \pmod {p}$.
	On the other hand, $\varphi(\pi^e)|\varphi(p)=p$. It follows that $p\varphi(\pi^{e-a})|\varphi(\pi^e)|p$. But $\varphi(\pi^{e-a})$ is nilpotent in $\pi_0(A/p^2)$, so we learn that $p=0$ in $\pi_0(A/p^2)$, and hence also in the subring $\pi_0(A)/p^2$. 
	This means that $p=p^2x$ for some $x \in \pi_0(A)$. Applying $\delta$, we obtain that 
	$\delta(p) = \delta(p^2x)$ in $\pi_0(A/p)$. However $\delta(p^2x)$ is divisible by 
	$p$ and $\delta(p)$ is a unit, so we learn that $A/p=0$, and hence $A=0$.
\end{proof}

\begin{proof}[Proof of \Cref{thm:optimalroots}]
	The fact that $r_{p,1,4} = 0$ for $p>2$ and $r_{2,1,4} = 1$ follows from \Cref{thm:ramificationobstruction} since $\Z_p[\zeta_p]$ and $\Z_2[\zeta_4]$ come from ramified extensions of $\QQ_p$ and $\QQ_2$ respectively.

	We now show that if we fix $p$ and $h$, then for $n\gg 0$, the number $r_{p,h,n}$ is $0$ if $p>2$ and is $1$ if 
	$p=2$. If this were false, then there woud be a sequence $R_n$ of nonzero $T(h)$-local 
	$\EE_n$-algebras containing a $p^{\text{th}}$ (resp. $4^{\text{th}}$ for $p=2$) root of 
	unity. The ultraproduct of these would then be a nonzero $T(h)$-local 
	$\EE_{\infty}$-algebra with such a root of unity, contradicting \Cref{thm:einfty}.
\end{proof}

\bibliographystyle{alpha}
\bibliography{main}
\end{document}